\documentclass{amsart}

\usepackage{hyperref}
\usepackage{mdwlist} 
\usepackage{enumerate}
\usepackage[abbrev]{amsrefs}
\usepackage{tikz}
\usepackage{xcolor} 
\newcommand\faint[1]{{\color{lightgray} #1}}

\newtheorem{lemma}{Lemma}
\newtheorem{theorem}[lemma]{Theorem}
\newtheorem{corollary}[lemma]{Corollary}

\title[Miniumum-weight perfect matchings on the line]%
{Minimum-weight perfect matching for~non-intrinsic distances on the
  line}
\author[J. Delon]{Julie Delon}
\address{LTCI CNRS\\
  TELECOM ParisTech\\
  46 rue Barrault\\
  75634 Paris\\
  France}
\author[J. Salomon]{Julien Salomon}
\address{CEREMADE\\
  Universit\'e Paris-Dauphine\\
  Place du Marechal Lattre de Tassigny\\
  75775 Paris\\
  France}
\author[A. Sobolevski]{Andrei Sobolevski}
\address{Institute for Information Transmission Problems (Kharkevich
  Institute)\\
  19 B. Karetny per.\\
  Moscow 127994\\
  Russia}
\address{Laboratoire J.-V.~Poncelet (UMI~2615 CNRS)\\
  11 B. Vlassievski per.\\
  Moscow 119002\\
  Russia
}
\email{sobolevski@iitp.ru}
\thanks{Supported by Agence Nationale de la Recherche under project
  ANR07--01--0235 OTARIE and the Russian Foundation for Basic Research
  under project 07--01--92217-CNRSL-a.}

\begin{document}

\begin{abstract}
  We consider a minimum-weight perfect matching problem on the line
  and establish a ``bottom-up'' recursion relation for weights of
  partial minimum-weight matchings. %
\end{abstract}

\maketitle

\section{Introduction}
\label{sec:introduction}

We start with recalling a few notions from combinatorial optimization
on graphs. %
A \emph{matching} in an undirected graph is any set of its mutually
disjoint edges: no two edges from such set can share a vertex. %
A matching is called \emph{perfect} if it involves all vertices of the
graph (the number of vertices is then necessarily even). %
  
Depending on the structure of the graph, perfect matchings may be
many. %
Suppose that edges of a graph are endowed with real \emph{weights};
then it makes sense to look for a perfect matching composed from a set
of edges with a minimum sum of weights. %
In this note we treat a particular case of this \emph{mini\-mum-weight
  perfect matching} problem where the graph is complete, all its
vertices are located on a line, and edge weights are related to
distances between the vertices along the line. %

A bipartite version of this problem, in which vertices are divided
into two equal classes and edges of the matching must connect vertices
of different class, reduces to transport optimization. %
For a particular class of cost functions (those of ``concave type'')
this problem has been thoroughly treated in the measure-theoretic
setting in~\cite{McCann.R:1999}. %
Similar problems have also long been considered in the algorithmics
literature for the specific case of the distance $|x - y|$, assuming
that the measures are discrete
\cites{Karp.R:1975,Werman.M:1986,Aggarwal.A:1992}. %
An algorithm for a general cost function of a concave type has been
proposed recently in~\cites{Delon.J:2010,Delon.J:2011}. %

In the present note we consider the minimum-weight perfect matching
problem on a complete graph without assuming it to be bipartite. %
However, for the class of weight functions generated by distances, the
non-bipartite problem turns out to be essentially equivalent to a
bipartite problem with alternating points due to a ``no-crossing''
property of the optimal matching. %
This observation, which we owe to S.~Nechaev, allows to employ the
theory developed for the bipartite case
in~\cites{Delon.J:2010,Delon.J:2011} to the non-bipartite matching. %

The main contribution of the present note is a specific ``bottom-up''
recursion relation~\eqref{eq:7} for partial minimum-weight
matchings. %
This recursion follows from a ``localization'' property of
minimum-cost perfect matchings for concave cost functions
(Theorem~\ref{stabilization}) and provides a new perspective on the
construction of~\cites{Delon.J:2010,Delon.J:2011}. %

The note is organized as follows. %
Section~\ref{sec:metric} is a review of the basic construction of
metrics on the real line. %
In Section~\ref{sec:matching-problem} we recall the problem of
minimum-weight perfect matching and cite a few useful results about
the structure of its solution. %
Section~\ref{sec:stitching-lemma} contains the key technical result, a
kind of localization principle for minimum-weight matching. %
Finally in Section~\ref{sec:bellman-equation} we derive the recursive
relation for partial minimum-weight matchings. %
We also discuss the ensuing algorithm and compare it to modern
variants of the Edmonds blossom algorithm for the minimum-weight
perfect matching problem. %

It is our pleasure to thank the organizers of the conference
\textit{Optimization and Stochastic Methods for Spatially Distributed
  Information} (St~Petersburg, EIMI, May 2010), where an earlier
version of this work was presented. %

\section{Intrinsic and non-intrinsic distances on the line}
\label{sec:metric}

Recall the usual axioms for a distance $d(\cdot, \cdot)$ on the real
line~$\mathbf{R}$: for all $x, y, z \in \mathbf{R}$,
\begin{enumerate}[({D}1)]
\item\label{distnonneg} $d(x, y) \ge 0$ with $d(x, y) = 0$ iff $x
  = y$;
\item\label{distsymm} $d(x, y) = d(y, x)$;
\item\label{disttriang} $d(x, y) + d(y, z) \ge d(x, z)$. %
\suspend{enumerate} 
These axioms are satisfied by the distance $d(x, y) = |x - y|$ as well
as by any distance of the form
\begin{equation}
  \label{eq:1}
  d_g(x, y) = g(|x - y|),
\end{equation}
where $g(\cdot)$ is a nonnegative concave function defined for all $x
\ge 0$ such that $g(x) = 0$ iff $x = 0$. %
Here concavity means that $g(\lambda x + (1 - \lambda) y) \ge \lambda
g(x) + (1 - \lambda) g(y)$ for all $x, y\ge 0$ and all $\lambda \in
[0, 1]$. %
Note that the distance $d_g$ is \emph{homogeneous} with respect to
shifts:
\resume{enumerate}[{[({D}1)]}]
\item\label{disthomo} $d_g(x + t, y + t) = d_g(x, y)$ for all~$x, y, t
  \in \mathbf{R}$. %
\end{enumerate}

Conversely, any distance satisfying the axioms
(D\ref{distnonneg})--(D\ref{disthomo}) has the form \eqref{eq:1} for a
suitable nonnegative concave function~$g$: indeed, take $g(x) = d(0,
x)$ for $x \ge 0$ and check that concavity follows
from~(D\ref{disttriang}). %

An example of this construction is the distance $|x - y|^\alpha$ with
$0 \le \alpha \le 1$, where
\begin{displaymath}
   |x - y|^0 = \begin{cases} 0, & x = y, \\ 1, & x\neq y\end{cases}\qquad
\end{displaymath}
is the ``discrete distance''; here we are mostly interested in the
case $0 < \alpha < 1$. %

An important property of a distance~$d(\cdot, \cdot)$ is whether it is
\emph{intrinsic}. %
To recall the corresponding definition, take two distinct points $x, y
\in \mathbf{R}$ and connect them with parameterized curves taking
values $x(t)$ in~$\mathbf{R}$, i.e., suppose that $0\le t\le 1$, $x(0)
= x$, and $x(1) = y$. %
By the triangle inequality,
\begin{equation}
  \label{eq:2}
  d(x, y) \le \inf \sum_{0\le i<N} d(x(t_i), d(x(t_{i + 1})),
\end{equation}
where the infimum is taken over all curves connecting $x$ to~$y$ and
all meshes $0 = t_0 < t_1 < t_2 < \dots < t_N = 1$ with $N \ge 1$. %
The distance~$d$ is called intrinsic if~\eqref{eq:2} is an equality
for all $x, y$ (see, e.g., \cite{Burago.D:2001}). %

The distance $|x - y|$ and its scalar multiples are the only
homogeneous intrinsic distances in~$\mathbf{R}$. %
For distances $|x - y|^\alpha$ with $0 < \alpha < 1$, connecting
curves have infinite length unless $x = y$, and therefore these
distances are not intrinsic. %

The intuition behind this notion is that whereas the geometry
of~$\mathbf{R}$ equipped with an intrinsic distance is fully
determined by the distance, the same line~$\mathbf{R}$ with a
non-intrinsic distance should be viewed as embedded in an auxiliary
space of a larger dimension, and geometry on~$\mathbf{R}$ is induced
by that in the embedding space. %
According to the Assouad embedding theorem (see, e.g.,
\cite{Heinonen.J:2001}), for the distance $|x - y|^\alpha$
on~$\mathbf{R}$ the embedding space has dimension of the
order~$1/\alpha$ for small $\alpha$. %

\section{Minimum-weight perfect matchings}
\label{sec:matching-problem}

Consider an even number of points $x_1 < x_2 < \dots < x_{2n}$ on the
real line~$\mathbf{R}$ equipped with a distance~$d$ and look for a
\emph{minimum-weight perfect matching} in a complete graph~$K_{2n}$ on
these points, where the weight of an edge connecting $x_i$ to~$x_j$
is~$d(x_i, x_j)$. %
It is convenient to represent $\mathbf{R}$ with a horizontal interval
and use arcs in the upper halfplane to show the edges of the
graph~$K_{2n}$. %

The following lemma, adapted from McCann
\cite{McCann.R:1999}*{Lemma~2.1}, allows to describe the structure of
minimum-weight matchings. %

\begin{lemma}
  \label{no-crossing}
  Suppose the distance has the form~$d_g$ \eqref{eq:1} with a
  \emph{strictly concave} function~$g$ (i.e., $g(\lambda x + (1 -
  \lambda) y) > \lambda g(x) + (1 - \lambda) g(y)$ for all $x, y\ge
  0$, $x \neq y$, and $0 < \lambda < 1$). %
  Then the inequality
  \begin{equation}
    \label{eq:3}
    d_g(x_1, y_1) + d_g(x_2, y_2) \le d_g(x_1, y_2) + d_g(x_2, y_1)
  \end{equation}
  implies that the intervals connecting $x_1$ to~$y_1$ and $x_2$
  to~$y_2$ are either disjoint or one of them is contained in the
  other. %
\end{lemma}

\begin{proof}
  Let arcs representing the matching be directed from $x$'s to~$y$'s
  and assume, without loss of generality, that $x_1$ is leftmost. %
  Then the configuration
  \begin{tikzpicture}
    \draw[thick] (0, 0) -- (1.75, 0);
    \draw[<-] (1, 0) arc (0:180:.375 and .25);
    \draw[->] (1.5, 0) arc (0:180:.375 and .25);
  \end{tikzpicture},
  where $x_1 < y_2 < y_1 < x_2$, cannot satisfy~\eqref{eq:3} because a
  strictly concave function~$g$ must be strictly growing and therefore
  $d_g(x_1, y_2) < d_g(x_1, y_1)$ and $d_g(x_2, y_1) < d_g(x_2,
  y_2)$. %
  To rule out configuration
  \begin{tikzpicture}
    \draw[thick] (0, 0) -- (1.75, 0);
    \draw[<-] (1, 0) arc (0:180:.375 and .25);
    \draw[<-] (1.5, 0) arc (0:180:.375 and .25);
  \end{tikzpicture},
  where $x_1 < x_2 < y_1 < y_2$, choose $0 < \lambda < 1$ such that
  \begin{displaymath}
    y_1 - x_1 = (1 - \lambda)(y_2 - x_1) + \lambda(y_1 - x_2).
  \end{displaymath}
  Then, since $(y_1 - x_1) + (y_2 - x_2) = (y_2 - x_1) + (y_1 - x_2)$,
  we obtain
  \begin{displaymath}
    y_2 - x_2 = \lambda(y_2 - x_1) + (1 - \lambda)(y_1 - x_2)
  \end{displaymath}
  and can further use the strict concavity of~$g$ to get
  \begin{displaymath}
    \begin{gathered}
      d_g(x_1, y_1) > (1 - \lambda)d_g(x_1, y_2) + \lambda d_g(x_2, y_1),\\
      d_g(x_2, y_2) > \lambda d_g(x_1, y_2) + (1 - \lambda)d_g(x_2, y_1).
    \end{gathered}
  \end{displaymath}
  The sum of these inequalities contradicts~\eqref{eq:3}. %
  All the other configurations where the points $x_2$, $y_1$, $y_2$
  are located to the right of~$x_1$, namely \begin{tikzpicture}
    \draw[thick] (0, 0) -- (1.75, 0); \draw[<-] (.75, 0) arc
    (0:180:.25); \draw[<-] (1.5, 0) arc (0:180:.25);
  \end{tikzpicture},
  \begin{tikzpicture}
    \draw[thick] (0, 0) -- (1.75, 0);
    \draw[<-] (.75, 0) arc (0:180:.25);
    \draw[->] (1.5, 0) arc (0:180:.25);
  \end{tikzpicture},
  \begin{tikzpicture}
    \draw[thick] (0, 0) -- (1.75, 0);
    \draw[<-] (1, 0) arc (0:180:.125);
    \draw[<-] (1.5, 0) arc (0:180:.625 and .25);
  \end{tikzpicture},
  \begin{tikzpicture}
    \draw[thick] (0, 0) -- (1.75, 0);
    \draw[->] (1, 0) arc (0:180:.125);
    \draw[<-] (1.5, 0) arc (0:180:.625 and .25);
  \end{tikzpicture}, are consistent with~\eqref{eq:3}. %
\end{proof}

It is possible to show that conversely, for an arbitrary bivariate
function $d(x, y)$ the property of Lemma~\ref{no-crossing} together
with homogeneity~(D\ref{disthomo}) imply that $d$ has the
form~\eqref{eq:1} for a suitable concave function~$g$
\cite{McCann.R:1999}*{Lemma~B4}. %

We call a matching \emph{nested} if, for any two arcs $(x_i, x_j)$ and
$(x_{i'}, x_{j'})$ that are present in the matching, the corresponding
intervals in~$\mathbf{R}$ are either disjoint or one of them is
contained in the other. %

\begin{theorem}[\cites{Aggarwal.A:1992,McCann.R:1999}]
  \label{nested}
  A minimum-weight matching is nested. %
\end{theorem}

\begin{proof}
  This statement essentially follows from Lemma~\ref{no-crossing},
  because whenever two arcs $(x_i, x_j)$ and~$(x_{i'}, x_{j'})$ are
  crossed, the sum of distances corresponding to the ``uncrossed''
  arcs $(x_i, x_{j'})$, $(x_{i'}, x_j)$ must be smaller according
  to~\eqref{eq:3}. %

  However uncrossing may introduce new crossings with other arcs, and
  it remains to be checked that it does decrease the total number of
  crossings. %
  This is done using an argument adapted from
  \cite{Aggarwal.A:1992}*{Lemma~1}. %

  Let an arc $(x_k, x_\ell)$ cross either of the arcs $(x_i, x_{j'})$
  or $(x_{i'}, x_j)$ obtained after uncrossing the arcs $(x_i, x_j)$
  and $(x_{i'}, x_{j'})$; we would like to prove that before
  uncrossing the arc~$(x_k, x_\ell)$ had the same number of crossings
  with $(x_i, x_j)$ and $(x_{i'}, x_{j'})$. %
  It suffices to consider the following two cases. %

  \textsc{Case $x_i < x_{j'} < x_j < x_{i'}$}. %
  Assume that $(x_k, x_\ell)$ crosses $(x_i, x_{j'})$. %
  Then $(x_k, x_\ell)$, shown as a dotted arc, may be situated in one of
  the following ways: %
  \begin{center}
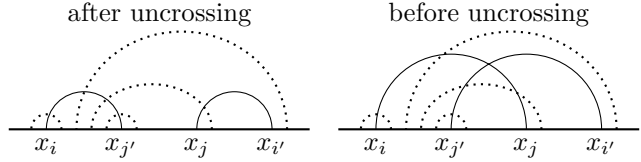
    \begin{tabular}{cc}
    after uncrossing & before uncrossing \\
    \begin{tikzpicture}
      \draw[thick] (0, 0) -- (4, 0);
      \draw (.5, 0) node [below] {$x_i$} (1.5, 0) node [below] {$x_{j'}$}
      (2.5, 0) node [below] {$x_j$} (3.5, 0) node [below] {$x_{i'}$};
      \draw (1.5, 0) arc (0:180:.5) (3.5, 0) arc (0:180:.5);
      \draw[thick, dotted] (.7, 0) arc (0:180:.2)
      (1.7, 0) arc (0:180:.2)
      (2.7, 0) arc (0:180:.8 and .6)
      (3.7, 0) arc (0:180:1.4 and 1.3);
    \end{tikzpicture} &
    \begin{tikzpicture}
      \draw[thick] (0, 0) -- (4, 0);
      \draw (.5, 0) node [below] {$x_i$} (1.5, 0) node [below] {$x_{j'}$}
      (2.5, 0) node [below] {$x_j$} (3.5, 0) node [below] {$x_{i'}$};
      \draw (2.5, 0) arc (0:180:1) (3.5, 0) arc (0:180:1);
      \draw[thick, dotted] (.7, 0) arc (0:180:.2)
      (1.7, 0) arc (0:180:.2)
      (2.7, 0) arc (0:180:.8 and .6)
      (3.7, 0) arc (0:180:1.4 and 1.3);
    \end{tikzpicture}
    \end{tabular}
  \end{center}
  We see that the number of crossings between $(x_k, x_\ell)$ and the
  other two arcs is the same before and after uncrossing. %
  The case where $(x_k, x_\ell)$ crosses $(x_{i'}, x_j)$ is
  symmetrical. %

  \textsc{Case $x_i < x_{i'} < x_j < x_{j'}$}. %
  The following situations of the dotted arc $(x_k, x_\ell)$ are
  possible:
  \begin{center}
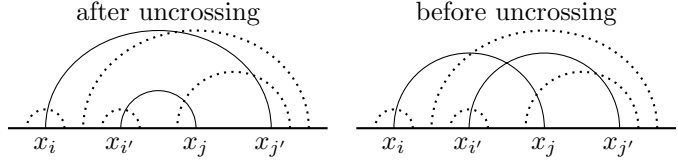
    \begin{tabular}{cc}
      after uncrossing & before uncrossing \\
      \begin{tikzpicture}
      \draw[thick] (0, 0) -- (4.25, 0);
      \draw (.5, 0) node [below] {$x_i$} (1.5, 0) node [below] {$x_{i'}$}
      (2.5, 0) node [below] {$x_j$} (3.5, 0) node [below] {$x_{j'}$};
      \draw (2.5, 0) arc (0:180:.5) (3.5, 0) arc (0:180:1.5 and 1.3);
      \draw[thick, dotted] (.75, 0) arc (0:180:.25)
      (1.75, 0) arc (0:180:.25)
      (3.75, 0) arc (0:180:.75)
      (4, 0) arc (0:180:1.5 and 1.3);
    \end{tikzpicture} &
    \begin{tikzpicture}
      \draw[thick] (0, 0) -- (4.25, 0);
      \draw (.5, 0) node [below] {$x_i$} (1.5, 0) node [below] {$x_{i'}$}
      (2.5, 0) node [below] {$x_j$} (3.5, 0) node [below] {$x_{j'}$};
      \draw (2.5, 0) arc (0:180:1) (3.5, 0) arc (0:180:1);
      \draw[thick, dotted] (.75, 0) arc (0:180:.25)
      (1.75, 0) arc (0:180:.25)
      (3.75, 0) arc (0:180:.75)
      (4, 0) arc (0:180:1.5 and 1.3);
    \end{tikzpicture} 
    \end{tabular}
  \end{center}
  Again the number of crossings between $(x_k, x_\ell)$ and the other two
  arcs is the same before and after uncrossing. %

  It follows that each uncrossing removes exactly one crossing from
  the matching, and therefore any possible sequence of uncrossings
  leads in a finite number of steps to a nested matching with a
  strictly smaller weight. %
  In other words, it suffices to look for the minimum-weight matching
  only among the nested ones. %
\end{proof}

This result implies that the minimum-weight perfect matching problem
is essentially bipartite, and it is indeed the bipartite setting that
is considered in \citelist{\cite{Aggarwal.A:1992} \cite{Delon.J:2010}
  \cite{Delon.J:2011} \cite{Karp.R:1975} \cite{McCann.R:1999}
  \cite{Werman.M:1986}} %

\begin{corollary}
  \label{bipartite}
  In a minimum-weight perfect matching, points with even numbers are
  matched to points with odd numbers. %
\end{corollary}

\begin{proof}
  Suppose on the contrary that $x_i$ is matched to~$x_j$ with $i$
  and~$j$ both even or both odd. %
  The interval connecting $x_i$ to~$x_j$ contains an odd number $|i -
  j| - 1$ of points, one of which has to be matched to a point outside
  and thus cause a crossing with $(x_i, x_j)$. %
  Since a minimum-weight matching is nested, the result follows. %
\end{proof}

Observe finally the following simple form of the Bellman optimality
principle. %

\begin{lemma}
  \label{bellman}
  Any subset of arcs in a minimum-weight matching is itself a
  minimum-weight perfect matching on the set of endpoints of the arcs
  that belong to this subset.
\end{lemma}

\begin{proof}
  Indeed, if one could rematch these points achieving a smaller total
  weight, then the full original matching itself would not be
  minimum-weight: rematching the corresponding subset of arcs (and
  possibly uncrossing any crossed arcs that might result from the
  rematch) would give a matching with a strictly smaller weight. %
\end{proof}

\section{Preservation of hidden arcs}
\label{sec:stitching-lemma}

Call an arc $(x_i, x_j)$ in a nested matching \emph{exposed} if it is
not contained in any other arc, i.e., if there is no arc $(x_{i'},
x_{j'})$ with $x_i, x_j$ contained between $x_{i'}$ and~$x_{j'}$. %
We call all other arcs in a nested matching non-exposed or
\emph{hidden}. %
Intuitively, exposed arcs are those visible ``from above'' and hidden
arcs are those covered with exposed ones. %

Suppose $X = \{x_i\}_{1\le i\le 2n}$ with $x_1 < x_2 < \dots < x_{2n}$
and $X' = \{x'_{i'}\}_{1\le i'\le 2n'}$ with $x'_1 < x'_2 < \dots <
x'_{2n'}$ be two sets such that $x_{2n} < x'_1$, i.e., $X'$ lies to
the right of~$X$. %
We will refer to minimum-weight perfect matchings on $X$ and~$X'$ as
\emph{partial matchings} and to the minimum-weight perfect matching
on~$X \cup X'$ as \emph{joint matching}. %
The following result is closely related to properties of ``local
matching indicators'' introduced and studied
in~\cites{Delon.J:2010,Delon.J:2011}. %

\begin{theorem}
  \label{stabilization}
  Whenever an arc $(x_i, x_j)$ [respectively $(x'_{i'}, x'_{j'})$] is
  hidden in the partial matching on~$X$ [respectively on~$X'$], it
  belongs to the joint optimal matching and is hidden there too. %
\end{theorem}

Observe that exposed arcs in partial matchings are generally not
preserved in the joint matching: they may disappear altogether or
become hidden. %

\begin{proof}
  By contradiction, assume that some of hidden arcs in the partial
  matching on~$X$ do not belong to the joint matching. %
  Then there will be at least one exposed arc $(x_\ell, x_r)$ in the
  partial matching on~$X$ such that some points $x_i$ with $x_\ell <
  x_i < x_r$ are connected in the joint matching to points outside
  $(x_\ell, x_r)$. %

  Indeed, if points inside every exposed arc $(x_\ell, x_r)$ would be
  matched in the joint matching only among themselves, then by
  Lemma~\ref{bellman} their matching would be exactly the same as in
  the partial matching on~$X$, and therefore all hidden arcs between
  $x_\ell$ and~$x_r$ would be preserved in the joint matching. %

  \begin{figure}
    \centering
    \begin{tikzpicture}
      \draw (1, 0) -- (11, 0);
      \draw (6.5, 0) node [below] {$\mathstrut z'_{k'}$} arc (180:0:2)
      (10.5, 0) node [below] {$\mathstrut y'_{k'}$};
      \draw (7, 0) node [below] {$\mathstrut\cdots$} arc (180:0:1.25)
      (9.5, 0) node [below] {$\mathstrut\cdots$};
      \draw (7.5, 0) node [below] {$\mathstrut z'_1$} arc (180:0:.5)
      (8.5, 0) node [below] {$\mathstrut y'_1$};
      \draw (6, 0) node [below] {$\mathstrut z_k$} arc (0:180:2.25)
      (1.5, 0) node [below] {$\mathstrut y_k$}
      (5, 0) arc (0:180:1.5) (5.25, 0) node [below] {$\mathstrut\cdots$}
      (2.25, 0) node [below] {$\mathstrut\cdots$};
      \draw (4.5, 0) node [below] {$\mathstrut z_1$} arc (0:180:.75)
      (3, 0) node [below] {$\mathstrut y_1$};
      \draw[thick,dotted] (8, 0) node [below] {$\mathstrut x_r$}
      arc (0:180:2) (4, 0) node [below] {$\mathstrut x_\ell$};
    \end{tikzpicture}
    \caption{Notation used in the proof of
      Theorem~\protect\ref{stabilization}. %
      Note that in general $x_\ell \le z_1$ and $z'_1 \le x_r$, but in
      the figure these pairs of points are shown to be distinct.}
    \label{fig:stabil1}
  \end{figure}
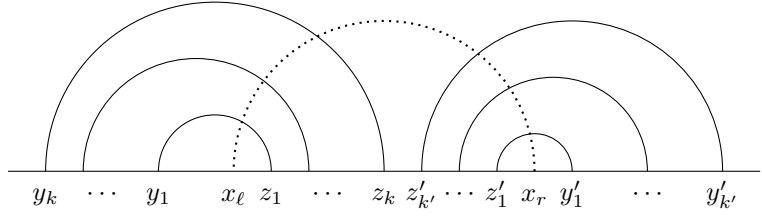
  Suppose $(x_\ell, x_r)$ is the leftmost arc of the above kind. %
  Denote all the points in the segment $[x_\ell, x_r]$ that are
  connected in the joint matching to points on the left of~$x_\ell$ by
  $z_1 < z_2 < \dots < z_k$; denote the opposite endpoints of the
  corresponding arcs by $y_1 > y_2 > \dots > y_k$, where the
  inequalities follow from the fact that the joint matching is
  nested. %
  Likewise denote those points from~$[x_\ell, x_r]$ that are connected
  in the joint matching to points on the right of~$x_r$ by $z'_1 >
  z'_2 > \dots > z'_{k'}$ and their counterparts in the joint matching
  by $y'_1 < y'_2 < \dots < y'_{k'}$ (fig.~\ref{fig:stabil1}). %

  Although $k$ or $k'$ may be zero, the number $k + k'$ must be
  positive and even. %
  Indeed, by Corollary~\ref{bipartite} the segment $[x_\ell, x_r]$
  contains an even number of points and all of them must be matched in
  a perfect matching; removing from the joint matching all arcs whose
  ends both lie in~$[x_\ell, x_r]$, we are left with an even number of
  points that are matched outside this segment. %

  Let us now restrict our attention to the segment $[x_\ell, x_r]$ and
  consider a matching that consists of the following arcs: those arcs
  of the joint matching whose both ends belong to~$[x_\ell, x_r]$; the
  arcs $(z_1, z_2)$, \dots, $(z_{2\kappa - 1}, z_{2\kappa})$, where%
  \footnote{$\lfloor \xi\rfloor$ is the largest integer $n$ such that
    $n\le \xi$.} %
  $\kappa = \lfloor k/2\rfloor$; the arcs $(z'_2, z'_1)$, \dots,
  $(z'_{2\kappa'}, z'_{2\kappa' - 1})$, where $\kappa' = \lfloor
  k'/2\rfloor$; and $(z_k, z'_{k'})$ if both $k$ and~$k'$ are odd. %

  Denote by~$W'$ the weight of this matching. %
  By Lemma~\ref{bellman}, it cannot be smaller than the weight~$W'_0$
  of the restriction of the partial matching on~$X$ to $[x_\ell,
  x_r]$. %
  For the total weight~$W$ of the joint matching on $X \cup X'$ we
  thus have
  \begin{equation}
    \label{eq:4}
    W \ge W - W' + W'_0.
  \end{equation}
  The right-hand side of~\eqref{eq:4} is represented in
  fig.~\ref{fig:stabil2}~(a) in the case when both $k$ and~$k'$ are
  odd. %
  It is a sum of positive terms corresponding to the arcs of the joint
  matching outside $[x_\ell, x_r]$ (not shown), the arcs of the
  partial matching on~$X$ inside $[x_\ell, x_r]$ (not shown, with
  exception of~$(x_\ell, x_r)$ represented with a solid arc in the
  upper halfplane), the arcs of the joint matching having one end
  inside $[x_\ell, x_r]$ and the other end outside this segment (solid
  arcs in the upper halfplane), and negative terms that come from
  subtraction of~$W'$ and correspond to the arcs connecting the $z$
  points (solid arcs in the lower halfplane). %

  We now show that by a suitable sequence of uncrossings the
  right-hand side of~\eqref{eq:4} can be further reduced to a matching
  whose weight is strictly less than~$W$. %

  \begin{figure}
    \centering
    \begin{tikzpicture}
      \draw (1, 0) -- (11, 0);
      \draw (6.5, 0) node [below] {$\mathstrut z'_{k'}$} arc (180:0:2)
      (10.5, 0) node [below] {$\mathstrut y'_{k'}$};
      \draw (7, 0) node [below] {$\mathstrut\cdots$} arc (180:0:1.25)
      (9.5, 0) node [below] {$\mathstrut\cdots$};
      \draw (7.5, 0) node [below] {$\mathstrut z'_1$} arc (180:0:.5)
      (8.5, 0) node [below] {$\mathstrut y'_1$};
      \draw (6, 0) node [below] {$\mathstrut z_k$} arc (0:180:2.25)
      (1.5, 0) node [below] {$\mathstrut y_k$}
      (5, 0) arc (0:180:1.5) (5.25, 0) node [below] {$\mathstrut\cdots$}
      (2.25, 0) node [below] {$\mathstrut\cdots$};
      \draw[very thick] (4.5, 0) node [below] {$\mathstrut z_1$} arc (0:180:.75)
      (3, 0) node [below] {$\mathstrut y_1$};
      \draw (4.5, 0) arc (180:360:.25 and .2) (6, 0) arc (180:360:.25
      and .2)
      (7, 0) arc (180:360:.25 and .2);
      \draw[very thick] (8, 0) node [below] {$\mathstrut x_r$}
      arc (0:180:2) (4, 0) node [below] {$\mathstrut x_\ell$};
    \end{tikzpicture}\\
    (a)\\[2ex]
    \begin{tikzpicture}
      \draw (1, 0) -- (11, 0);
      \draw (6.5, 0) node [below] {$\mathstrut z'_{k'}$} arc (180:0:2)
      (10.5, 0) node [below] {$\mathstrut y'_{k'}$};
      \draw (7, 0) node [below] {$\mathstrut\cdots$} arc (180:0:1.25)
      (9.5, 0) node [below] {$\mathstrut\cdots$};
      \draw (7.5, 0) node [below] {$\mathstrut z'_1$} arc (180:0:.5)
      (8.5, 0) node [below] {$\mathstrut y'_1$};
      \draw (6, 0) node [below] {$\mathstrut z_k$} arc (0:180:2.25)
      (1.5, 0) node [below] {$\mathstrut y_k$};
      \draw[very thick] (5, 0) arc (0:180:1.5)
      (5.25, 0) node [below] {$\mathstrut\cdots$}
      (2.25, 0) node [below] {$\mathstrut\cdots$};
      \draw (4.5, 0) node [below] {$\mathstrut z_1$};
      \draw (4, 0) arc (0:180:.5) (3, 0) node [below] {$\mathstrut y_1$};
      \draw[very thick] (4.5, 0) arc (180:360:.25 and .2);
      \draw (6, 0) arc (180:360:.25 and .2);
      \draw (7, 0) arc (180:360:.25 and .2);
      \draw[very thick] (8, 0) node [below] {$\mathstrut x_r$}
      arc (0:180:1.75) (4, 0) node [below] {$\mathstrut x_\ell$};
    \end{tikzpicture}\\
    (b)\\[2ex]
    \begin{tikzpicture}
      \draw (1, 0) -- (11, 0);
      \draw (6.5, 0) node [below] {$\mathstrut z'_{k'}$} arc (180:0:2)
      (10.5, 0) node [below] {$\mathstrut y'_{k'}$};
      \draw (7, 0) node [below] {$\mathstrut\cdots$} arc (180:0:1.25)
      (9.5, 0) node [below] {$\mathstrut\cdots$};
      \draw (7.5, 0) node [below] {$\mathstrut z'_1$} arc (180:0:.5)
      (8.5, 0) node [below] {$\mathstrut y'_1$};
      \draw (6, 0) node [below] {$\mathstrut z_k$} arc (0:180:2.25)
      (1.5, 0) node [below] {$\mathstrut y_k$};
      \draw (5, 0) node [below] {$\mathstrut\cdots$}
      (2.25, 0) node [below] {$\mathstrut\cdots$};
      \draw (4, 0) arc (0:180:.5) (3, 0) node [below] {$\mathstrut y_1$};
      \draw (6, 0) arc (180:360:.25 and .2);
      \draw (7, 0) arc (180:360:.25 and .2);
      \draw (8, 0) node [below] {$\mathstrut x_r$} arc (0:180:3 and 1.75)
      (4, 0) node [below] {$\mathstrut x_\ell$};
    \end{tikzpicture}\\
    (c)\\
    \caption{Step~1 of the proof (see explanation in the text).}
    \label{fig:stabil2}
  \end{figure}
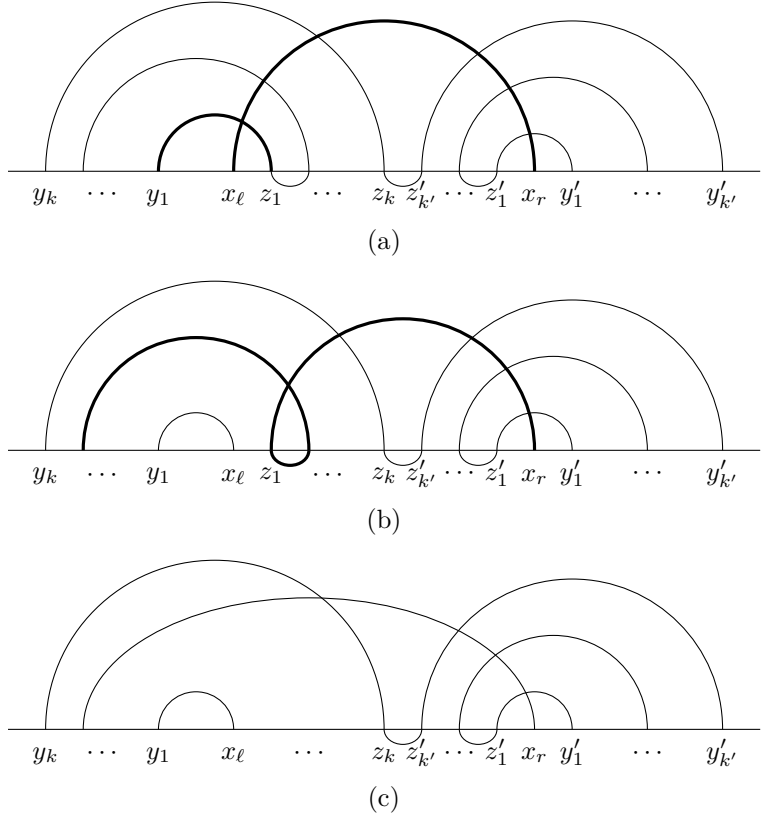

  \textsc{Step 1}. %
  Note that the arcs $(z_1, y_1)$ and~$(x_\ell, x_r)$, shown in
  fig.~\ref{fig:stabil2}~(a) with thick lines, are crossing. %
  Therefore
  \begin{displaymath}
    d(y_1, z_1) + d(x_\ell, x_r) > d(y_1, x_\ell) + d(z_1, x_r).
  \end{displaymath}
  Uncrossing these arcs gives the matching represented in
  fig.~\ref{fig:stabil2}~(b) and strictly reduces the right-hand side
  of~\eqref{eq:4}:
  \begin{displaymath}
    W > W - W' + W'_0 - d(y_1, z_1) - d(x_\ell, x_r) + d(y_1, x_\ell) +
    d(z_1, x_r).
  \end{displaymath}
  Now the arcs $(y_2, z_2)$ and $(z_1, x_r)$ are crossing, so
  \begin{displaymath}
    d(y_2, z_2) + d(z_1, x_r) - d(z_1, z_2) > d(y_2, x_r)
  \end{displaymath}
  and therefore
  \begin{displaymath}
    W > W - W' + W'_0 - d(y_1, z_1) - d(y_2, z_2) - d(x_\ell, x_r) +
    d(y_1, x_\ell) + d(z_1, z_2) + d(y_2, x_r).
  \end{displaymath}
  The right-hand side of this inequality is represented in
  fig.~\ref{fig:stabil2}~(c). %

  Repeating this step $\kappa = \lfloor k/2\rfloor$ times gives the
  inequality
  \begin{multline*}
    W > W - W' + W'_0 - d(x_\ell, x_r) - \sum_{1\le i\le 2\kappa} d(y_i, z_i) \\
    + \sum_{1\le i\le \kappa} d(z_{2i - 1}, z_{2i})
    + \sum_{1\le i\le \kappa} d(y_{2i - 1}, y_{2i - 2}) +
    d(y_{2\kappa}, x_r),
  \end{multline*}
  where in the rightmost sum $y_0$ is defined to be~$x_\ell$. %
  Note that at this stage all arcs coming to points $z_1, z_2, \dots$
  from outside $[x_\ell, x_r]$ are eliminated from the matching, except
  possibly $(y_k, z_k)$ if $k$ is odd. %

  \textsc{Step 2}. %
  It is now clear by symmetry that a similar reduction step can be
  performed on arcs going from $z'_1, z'_2, \dots$ to the right. %
  Repeating this $\kappa' = \lfloor k'/2\rfloor$ times gives the
  inequality
  \begin{multline*}
    W > W - W' + W'_0 - d(x_\ell, x_r) - \sum_{1\le i\le 2\kappa} d(y_i, z_i) 
    - \sum_{1\le i'\le 2\kappa'} d(z'_{i'}, y'_{i'}) \\
    + \sum_{1\le i\le \kappa} d(z_{2i - 1}, z_{2i})
    + \sum_{1\le i\le \kappa} d(y_{2i - 1}, y_{2i - 2}) \\
    + \sum_{1\le i'\le \kappa'} d(z'_{2i'}, z'_{2i' - 1})
    + \sum_{1\le i'\le \kappa'} d(y'_{2i' - 2}, y'_{2i' - 1}) +
    d(y_{2\kappa}, y'_{2\kappa'}),
  \end{multline*}
  where $y'_0 = x_r$. %

  \begin{figure}
    \centering
    \begin{tikzpicture}
      \draw (1, 0) -- (11, 0);
      \draw (6.5, 0) node [below] {$\mathstrut z'_{k'}$} arc (180:0:2)
      (10.5, 0) node [below] {$\mathstrut y'_{k'}$};
      \draw (7.25, 0) node [below] {$\mathstrut\cdots$}
      (9.5, 0) node [below] {$\mathstrut\cdots$};
      \draw (8, 0) node [below] {$\mathstrut x_r$} arc (180:0:.25)
      (8.5, 0) node [below] {$\mathstrut y'_1$};
      \draw[very thick] (6, 0) node [below] {$\mathstrut z_k$}
      arc (0:180:2.25) (1.5, 0) node [below] {$\mathstrut y_k$};
      \draw (5, 0) node [below] {$\mathstrut\cdots$}
      (2.25, 0) node [below] {$\mathstrut\cdots$};
      \draw (4, 0) arc (0:180:.5) (3, 0) node [below] {$\mathstrut y_1$};
      \draw (6, 0) arc (180:360:.25 and .2);
      \draw[very thick] (9.5, 0) arc (0:180:3.75 and 1.75)
      (4, 0) node [below] {$\mathstrut x_\ell$};
    \end{tikzpicture}\\
    (a)\\[2ex]
    \begin{tikzpicture}
      \draw (1, 0) -- (11, 0);
      \draw[very thick] (6.5, 0) node [below] {$\mathstrut z'_{k'}$}
      arc (180:0:2) (10.5, 0) node [below] {$\mathstrut y'_{k'}$};
      \draw (7.25, 0) node [below] {$\mathstrut\cdots$}
      (9.5, 0) node [below] {$\mathstrut\cdots$};
      \draw (8, 0) node [below] {$\mathstrut x_r$} arc (180:0:.25)
      (8.5, 0) node [below] {$\mathstrut y'_1$};
      \draw (6, 0) node [below] {$\mathstrut z_k$};
      \draw (1.5, 0) node [below] {$\mathstrut y_k$} arc (180:0:.25);
      \draw (5, 0) node [below] {$\mathstrut\cdots$}
      (2.25, 0) node [below] {$\mathstrut\cdots$};
      \draw (4, 0) arc (0:180:.5) (3, 0) node [below] {$\mathstrut y_1$};
      \draw (6, 0) arc (180:360:.25 and .2);
      \draw[very thick] (9.5, 0) arc (0:180:1.75)
      (4, 0) node [below] {$\mathstrut x_\ell$};
    \end{tikzpicture}\\
    (b)\\[2ex]
    \begin{tikzpicture}
      \draw (1, 0) -- (11, 0);
      \draw[very thick] (6.5, 0) node [below] {$\mathstrut z'_{k'}$}
      arc (0:180:.25) (10.5, 0) node [below] {$\mathstrut y'_{k'}$};
      \draw (7.25, 0) node [below] {$\mathstrut\cdots$}
      (9.5, 0) node [below] {$\mathstrut\cdots$};
      \draw (8, 0) node [below] {$\mathstrut x_r$} arc (180:0:.25)
      (8.5, 0) node [below] {$\mathstrut y'_1$};
      \draw (6, 0) node [below] {$\mathstrut z_k$};
      \draw (1.5, 0) node [below] {$\mathstrut y_k$} arc (180:0:.25);
      \draw (5, 0) node [below] {$\mathstrut\cdots$}
      (2.25, 0) node [below] {$\mathstrut\cdots$};
      \draw (4, 0) arc (0:180:.5) (3, 0) node [below] {$\mathstrut y_1$};
      \draw (6, 0) arc (180:360:.25 and .2);
      \draw (9.5, 0) arc (180:0:.5) (4, 0) node [below] {$\mathstrut x_\ell$};
    \end{tikzpicture}\\
    (c)\\
    \caption{Step~3 of the proof. %
      Note that in stage~(c) the arc $(z_k, z'_{k'})$ gives two
      contributions with positive and negative signs, which cancel
      out each other.}
    \label{fig:stabil3}
  \end{figure}
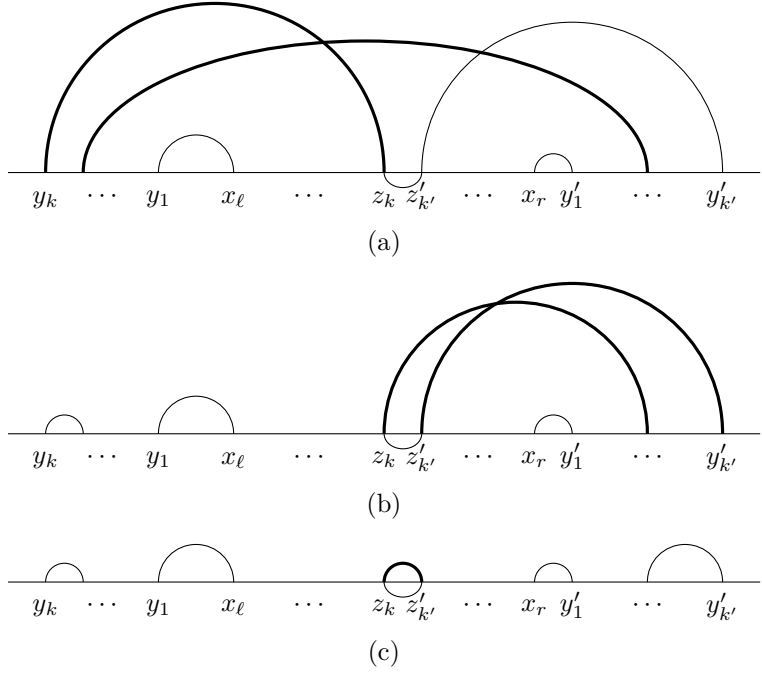

  \textsc{Step 3}. %
  If $k$ and $k'$ are odd, we perform two more uncrossings shown in
  fig.~\ref{fig:stabil3}. %

  The final estimate for~$W$ has the form
  \begin{multline}
    \label{eq:5}
    W > W - W' + W'_0 - d(x_\ell, x_r) - \sum_{1\le i\le k} d(y_i, z_i)
    - \sum_{1\le i'\le k'} d(z'_{i'}, y'_{i'}) \\
    + \sum_{1\le i\le \kappa} d(z_{2i - 1}, z_{2i})
    + \sum_{1\le i'\le \kappa'} d(z'_{2i'}, z'_{2i' - 1})
    + d(z_k, z'_{k'})\cdot [\text{$k$, $k'$ are odd}] \\
    + \sum_{1\le i\le\kappa} d(y_{2i - 1}, y_{2i - 2})
    + \sum_{1\le i'\le\kappa'} d(y'_{2i' - 2}, y'_{2i' - 1})
    + d(y_k, y'_{k'})\cdot[\text{$k$, $k'$ are even}],
  \end{multline}
  where notation such as [$k$, $k'$ are odd] means~$1$ if $k$, $k'$
  are odd and~$0$ otherwise. %

  The right-hand side of~\eqref{eq:5} contains four groups of terms:
  first,
  \begin{displaymath}
    W - \sum_{1\le i\le k} d(y_i, z_i) - \sum_{1\le i'\le k'} d(z'_{i'}, y'_{i'}),
  \end{displaymath}
  corresponding to the joint matching without the arcs connecting
  points inside $[x_\ell, x_r]$ to points outside this segment;
  second,
  \begin{displaymath}
    W' - \sum_{1\le i\le \kappa} d(z_{2i - 1}, z_{2i})
    - \sum_{1\le i'\le\kappa'} d(z'_{2i'}, z'_{2i' - 1})
    - d(z_k, z'_{k'})\cdot[\text{$k$, $k'$ are odd}],
  \end{displaymath}
  which comes with a negative sign and corresponds to the arcs of the
  joint matching with both ends inside~$[x_\ell, x_r]$, and cancels
  them from the total; third,
  \begin{displaymath}
    W'_0 - d(x_\ell, x_r),
  \end{displaymath}
  with positive sign, which corresponds to the hidden arcs of the
  partial matching on~$X$ inside the exposed arc $(x_\ell, x_r)$, not
  including the latter; %
  and finally the terms in the last line of~\eqref{eq:5},
  corresponding to the arcs matching $x_\ell$, $x_r$, and points $y_1,
  \dots, y_k, y'_1, \dots, y'_{k'}$, i.e., those points outside
  $[x_\ell, x_r]$ that were connected in the joint matching to points
  inside this segment. %

  Gathering together contributions of these four groups of terms, we
  observe that all negative terms cancel out and what is left
  corresponds to a perfect matching with a weight strictly smaller
  than~$W$, in which all arcs hidden by $(x_\ell, x_r)$ in the partial
  matching on~$X$ are restored. %
  There may still be some crossings caused by terms of the fourth
  group and \emph{not} involving the hidden arcs in $[x_\ell, x_r]$;
  uncrossing them if necessary gives a nested perfect matching whose
  weight is strictly less than that of the joint matching. %
  This contradicts the assumption that the latter is the
  minimum-weight matching on~$X\cup X'$. %
  Therefore all hidden arcs in the partial matching on~$X$ (and, by
  symmetry, those in the partial matching on~$X'$) belong to the joint
  matching. %
\end{proof}

\section{Recursion for minimum weights}
\label{sec:bellman-equation}

In this section we show how to apply Theorem~\ref{stabilization} to
compute the minimum-weight perfect matching algorithmically. %
Let $x_1 < x_2 < \dots < x_{2n}$ be a set of points on the real
line~$\mathbf{R}$ equipped with a homogeneous distance~$d$ of the
form~\eqref{eq:1}. %
For indices $i$,~$j$ of opposite parity and such that $i < j$, let
$W_{i,j}$ be the weight of the minimum-weight perfect matching on the
$j - i + 1$ points~$x_i < x_{i + 1} < \dots < x_j$. %
It is convenient to organize weights~$W_{i, j}$ with $i < j$ into a
pyramidal table:
\begin{displaymath}
  \begin{array}{cccccccll}
    &&&& W_{1, 2n} &&&& \\
    &&& \hdotsfor{3} &&& \\
    && W_{1, 6} & W_{2, 7} & W_{3, 8} & W_{4, 9} & \dots && \\
    & W_{1, 4} & W_{2, 5} & W_{3, 6} & W_{4, 7} & W_{5, 8} & \dots & W_{2n - 3, 2n}     & \\
    W_{1,2} & W_{2, 3} & W_{3, 4} & W_{4, 5} & W_{5, 6} & W_{6, 7} & \dots & W_{2n - 2, 2n - 1} & W_{2n - 1, 2n}
  \end{array}
\end{displaymath}
\begin{theorem}
  \label{recursion}
  For all indices $i$,~$j$ of opposite parity with $1 \le i < j \le
  2n$, weights~$W_{i, j}$ satisfy the following second-order recursion
  \begin{equation}
    \label{eq:7}
    W_{i, j} = \min\, \bigl[d(x_i, x_j) + W_{i + 1, j - 1},\
    W_{i, j - 2} + W_{i + 2, j} - W_{i + 2, j - 2}\bigr]
  \end{equation}
  with ``initial conditions''
  \begin{equation}
    \label{eq:6}
    W_{i, i - 1} = 0,\quad W_{i + 2, i - 1} = -d(x_i, x_{i + 1}).
  \end{equation}
\end{theorem}

\begin{proof}
  By an abuse of notation, we will refer to the minimum-weight
  perfect matching on points $x_r < x_{r + 1} < \dots < x_s$ as the
  ``matching~$W_{r, s}$.'' %

  Consider first the matching that consists of the arc~$(x_i, x_j)$
  and all arcs of the matching~$W_{i + 1, j - 1}$, and observe that by
  Lemma~\ref{bellman} its weight $d(x_i, x_j) + W_{i + 1, j - 1}$ is
  minimal among all matchings that contain~$(x_i, x_j)$. %

  We now examine the meaning of the expression $W_{i, j - 2} + W_{i +
    2, j} - W_{i + 2, j - 2}$. %
  Denote the point connected in the matching~$W_{i, j - 2}$ to~$x_i$
  by~$x_k$ and the point connected to~$x_{i + 1}$ by~$x_\ell$. %
  By Corollary~\ref{bipartite}, the pairs of indices $i, k$ and $i +
  1, \ell$ both have opposite parity. %
  Assume first that
  \begin{equation}
    \label{eq:13}
    x_{i + 1} < x_\ell < x_k \le x_{j - 2}.
  \end{equation}
  Applying Theorem~\ref{stabilization} to the sets $X = \{x_i, x_{i +
    1}\}$ and $X' = \{x_{i + 2}, \dots, x_{j - 2}\}$ and taking into
  account parity of $k$ and~$\ell$, we see that $x_k$ and~$x_\ell$ (as well
  as their neighbors $x_{k + 1}$ and~$x_{\ell - 1}$ if they are contained
  in~$[x_{i + 2}, x_{j - 2}]$) belong to exposed arcs of the
  matching~$W_{i + 2, j - 2}$. %
  Thus the matching~$W_{i, j - 2}$ has the following structure:
  
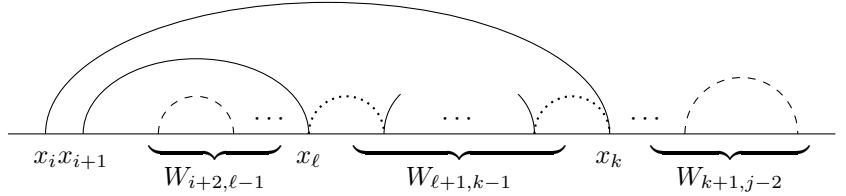
\begin{center}
    \begin{tikzpicture}
      \draw (0, 0) -- (11, 0) node [right] {,};
      \draw (.5, 0) node [below] {$\mathstrut x_i$}
      arc (180:0:3.75 and 1.75) (8, 0) node [below] {$\mathstrut x_k$};
      \draw (1, 0) node [below] {$\mathstrut x_{i + 1}$}
      arc (180:0:1.5 and 1) (4, 0) node [below] {$\mathstrut x_\ell$};
      \draw[dashed] (2, 0) arc (180:0:.5);
      \draw (2.75, 0) node [below] {$\underbrace{\qquad\qquad\quad}%
        _{\displaystyle W_{i + 2, \ell - 1}}$};
      \draw (3.5, 0) node [above] {$\cdots$};
      \draw[thick,dotted] (4, 0) arc (180:0:.5) (7, 0) arc (180:0:.5);
      \draw (5, 0) arc (180:135:.75) (7, 0) arc (0:45:.75);
      \draw (6, 0) node [above] {$\cdots$}
      node [below] {$\underbrace{\qquad\qquad\qquad\qquad}%
        _{\displaystyle W_{\ell + 1, k - 1}}$};
      \draw (8.5, 0) node [above] {$\cdots$};
      \draw[dashed] (9, 0) arc (180:0:.75);
      \draw (9.6, 0) node [below] {$\underbrace{\qquad\qquad\qquad}%
        _{\displaystyle W_{k + 1, j - 2}}$};
    \end{tikzpicture}
  \end{center}
  where dashed (resp., dotted) arcs correspond to those exposed arcs
  of the matching~$W_{i + 2, j - 2}$ that belong (resp., do not
  belong) to~$W_{i, j - 2}$. %

  Since points $x_{\ell - 1}$ and~$x_{k + 1}$ belong to exposed arcs in
  the matching~$W_{i + 2, j - 2}$, by Lemma~\ref{bellman} we see that
  the (possibly empty) parts of this matching that correspond to
  points $x_{i + 2} < \dots < x_{\ell - 1}$ and $x_{k + 1} < \dots < x_{j
    - 2}$ coincide with the (possibly empty) matchings $W_{i + 2, \ell -
    1}$ and~$W_{k + 1, j - 2}$. %
  For the same reason the (possibly empty) part of the matching~$W_{i,
    j - 2}$ supported on $x_{\ell + 1} < \dots < x_{k - 1}$ coincides
  with~$W_{\ell + 1, k -1}$. %
  Therefore
  \begin{equation}
    \label{eq:8}
    W_{i, j - 2} = d(x_i, x_k) + d(x_{i + 1}, x_\ell) + W_{i + 2, \ell - 1}
    + W_{\ell + 1, k - 1} + W_{k + 1, j - 2}.
  \end{equation}
  Taking into account~\eqref{eq:6}, observe that in the case $k = i +
  1$ and $\ell = i$, which was left out in~\eqref{eq:13}, this expression
  still gives the correct formula $W_{i, j - 2} = d(x_i, x_{i + 1}) +
  W_{i + 2, j - 2}$. %

  Now assume that in the matching $W_{i + 1, j}$ the point $x_j$ is
  connected to $x_{\ell'}$ and the point~$x_{j - 1}$ to~$x_{k'}$. %
  A similar argument gives
  \begin{equation}
    \label{eq:10}
    W_{i + 2, j} = W_{i + 2, \ell' - 1} + W_{\ell' + 1, k' - 1} + W_{k' + 1,
      j - 2} + d(x_{\ell'}, x_j) + d(x_{k'}, x_{j - 1});
  \end{equation}
  in particular, if $\ell' = j - 1$ and $k' = j$, then $W_{i + 2, j} =
  W_{i + 2, j - 2} + d(x_{j - 1}, x_j)$. %

  Suppose that $x_k < x_{\ell'}$. %
  Using again Lemma~\ref{bellman} and taking into account that $x_k$,
  $x_{k + 1}$, $x_{\ell' - 1}$, and~$x_{\ell'}$ all belong to exposed
  arcs in~$W_{i + 2, j - 2}$, we can write
  \begin{equation}
    \label{eq:12}
    W_{k + 1, j - 2} = W_{k + 1, \ell' - 1} + W_{\ell', j - 2},\quad
    W_{i + 2, \ell' - 1} = W_{i + 2, k} + W_{k + 1, \ell' - 1}
  \end{equation}
  and
  \begin{equation}
    \label{eq:11}
    W_{i + 2, j - 2} = W_{i + 2, k} + W_{k + 1, \ell' - 1} + W_{\ell', j - 2}.
  \end{equation}
  Substituting~\eqref{eq:12} into \eqref{eq:8} and~\eqref{eq:10} and
  taking into account~\eqref{eq:11}, we obtain
  \begin{multline*}
    W_{i, j - 2} + W_{i + 2, j} - W_{i + 2, j - 2} = d(x_i, x_k) +
    d(x_{i + 1}, x_\ell) + W_{i + 2, \ell - 1} + W_{\ell + 1,
      k - 1} \\
    + W_{k + 1, \ell' - 1} + d(x_{\ell'}, x_j) + W_{\ell' + 1, k' - 1} +
    d(x_{k'}, x_{j - 1}) + W_{k' + 1, j - 2}.
  \end{multline*}
  The right-hand side of this expression corresponds to a matching
  that coincides with $W_{i, j - 2}$ on $[x_i, x_k]$, with $W_{i + 2,
    j - 2}$ on~$[x_{k + 1}, x_{\ell' - 1}]$, and with~$W_{i + 1, j}$ on
  $[x_{\ell'}, x_j]$. %
  By Lemma~\ref{bellman} this matching cannot be improved on any of
  these three segments and is therefore optimal among all matchings in
  which $x_i$ and~$x_j$ belong to different exposed arcs. %

  It follows that under the assumption that $x_k < x_{\ell'}$ the
  expression in the right-hand side of~\eqref{eq:7} gives the minimum
  weight of all matchings on $x_i < x_{i + 1} < \dots < x_j$. %
  Moreover, the only possible candidates for the optimal matching are
  those constructed above: one that corresponds to $d(x_i, x_j) + W_{i
    + 1, j - 1}$ and one given by the right-hand side of the latter
  formula. %

  It remains to consider the case $x_k \ge x_{\ell'}$. %
  Since $x_k \neq x_{\ell'}$ for parity reasons, it follows that $x_k >
  x_{\ell'}$; now a construction similar to the above yields a matching
  which corresponds to $W_{i, j - 2} + W_{i + 2, j} - W_{i + 2, j -
    2}$ and in which the arcs $(x_i, x_k)$ and~$(x_{\ell'}, x_j)$ are
  crossed. %
  Uncrossing them leads to a matching with strictly smaller weight,
  which contains the arc~$(x_i, x_j)$ and therefore cannot be better
  than $d(x_i, x_j) + W_{i + 1, j - 1}$. %
  This means that~\eqref{eq:7} holds in this case too with $W_{i, j} =
  d(x_i, x_j) + W_{i + 1, j - 1}$. %
\end{proof}

Obviously, recursion~\eqref{eq:7} can be solved for all $1\le i < j
\le 2n$ in $O(n^2)$ operations, resulting in computation of
weights~$W_{ij}$ of all partial optimal matchings. %
This process is carried out in a ``bottom to top'' fashion: in the
pyramid, weights~$W_{i, j}$ with smaller values of~$j - i$ are
computed first. %

To determine the optimal matching on all the points $x_1, x_2, \dots,
x_{2n}$, one should keep track of those pairs $(i, j)$ for which
minimum in~\eqref{eq:6} is attained at the first alternative. %
Indeed, if the minimum is never attained at the first alternative,
then it is easy to see that the optimal perfect matching is $(x_1,
x_2), (x_3, x_4), \dots, (x_{2n - 1}, x_{2n})$. %
Suppose now that the first alternative provides minimum for for some
$W_{i_0j_0}$. %
Then according to Theorem~\ref{stabilization} one can retain the
matching for the points $x_{i_0 + 1}, \dots, x_{j_0 - 1}$ that has
been computed by this moment and consider a new, smaller
minimum-weight perfect matching problem on the points $x_1, x_2,
\dots, x_{i_0}, x_{j_0}, x_{j_0 + 1}, \dots, x_{2n}$. %
More precisely, it suffices to remove from the pyramidal table
quantities $W_{ij}$ with $i, j$ satisfying at least one of the
conditions $i_0 < i < j_0$ or $i_0 < j < j_0$, replace $W_{i_0j_0}$
with~$d(x_{i_0}, x_{j_0})$, stack the cells with either $i$ or~$j$
outside $(i_0, j_0)$ into a smaller pyramidal table, and continue
solving the recursion. %

\begin{figure}
  \begin{displaymath}
    \begin{array}{ccccccccc}
      &&&& (W_{1, 10}) &&&& \\
      &&& (W_{1, 8}) & (W_{2, 9}) & (W_{3, 10}) &&& \\
      && (W_{1, 6}) & (W_{2, 7}) & (W_{3, 8}) & \faint{(W_{4, 9})} & \faint{(W_{5, 10})} && \\
      & \faint{W_{1, 4}} & \faint{W_{2, 5}} & [W_{3, 6}] & \faint{(W_{4, 7})} & \faint{(W_{5, 8})} & (W_{6, 9}) & (W_{7, 10}) & \\
      W_{1, 2} & W_{2, 3} & \faint{W_{3, 4}} & \faint{W_{4, 5}} & \faint{W_{5, 6}} & W_{6, 7} & W_{7, 8} & W_{8, 9} & W_{9, 10} \\[1ex]
      &&&& \downarrow &&&& \\[1ex]
      &&&& (W_{1, 10}) &&&& \\
      &&& (W_{1, 8}) & (W_{2, 9}) & (W_{3, 10}) &&& \\
      && (W_{1, 6}) & (W_{2, 7}) & (W_{3, 8}) & (W_{6, 9}) & (W_{7, 10}) && \\
      & W_{1, 2} & W_{2, 3} & d(x_3, x_6) & W_{6, 7} & W_{7, 8} & W_{8, 9} & W_{9, 10} & \\
    \end{array}
  \end{displaymath}
  \caption{Reduction of the pyramidal table ($i_0 = 3$, $j_0 = 6$, $n
    = 5$).}
  \label{fig:reduction}
\end{figure}

This reduction step is illustrated in fig.~\ref{fig:reduction}. %
The elements to be removed from the table are shown in light gray in
the top pane. %
Assuming that the rows are scanned left to right, at the moment when
it is found that the element $W_{i_0, j_0} = W_{3, 6}$ (in brackets)
involves the first alternative in~\eqref{eq:7}, the elements shown in
parentheses have not yet been computed. %
Those of them that are to be kept in the table (the ``black'' ones)
are then stuck with the already computed elements to form a smaller
pyramid shown in the bottom, and the recursion resumes with the
element $W_{3, 6}$ replaced with $d(x_3, x_6)$. %

At the reduction step illustrated in fig.~\ref{fig:reduction}, the
matching is updated with one arc $(x_4, x_5)$, which is guaranteed to
belong to the optimal matching by Theorem~\ref{stabilization}. %
Generally, every time a reduction step is performed on an
element~$W_{i_0, j_0}$ and results in removal of the points $x_{k_1} <
x_{k_2} < \dots < x_{k_{2m}}$ that have been ``covered'' with the
arc~$(i_0, j_0)$, the matching is updated with arcs $(x_{k_1},
x_{k_2}), \dots, (x_{k_{2m - 1}}, x_{k_{2m}})$. %
When the process reaches the top element of the pyramid, a similar
matching of consecutive pairs is performed on the remaining
elements. %

A formal description of this procedure is given in
\cites{Delon.J:2010,Delon.J:2011} in somewhat different terms
(emphasizing ``local matching indicators'' instead of the
quantities~$W_{ij}$). %
Since this algorithm requires computing all the weights $W_{ij}$ in
the worst case, it takes at most~$O(n^2)$ operations for the matching
on~$2n$ points.

It is interesting to compare performance of our algorithm to the
celebrated Edmonds algorithm for minimum-weight perfect matching in a
general graph with $N$ vertices and~$M$ edges
\cites{Edmonds.J:1965a,Edmonds.J:1965b}. %
Note that we effectively consider a complete graph on $N = 2n$
vertices, i.e., $M = n(2n - 1)$. %
In this case recent improvements of the Edmonds algorithm run in
$O(n^3)$ time (see, e.g., the survey~\cite{Cook.W:1999}) whereas our
specialized algorithm is at most quadratic in~$n$. %

\end{document}